\documentclass[a4paper,12pt]{amsart}
\usepackage{amsmath, amsthm, amssymb, amsopn, amscd, mathtools, array, enumerate, amsfonts, graphics, hyperref, wasysym, chngcntr, linguex, graphicx, ytableau,youngtab,young,tikz}

\newtheorem{theorem}{Theorem}[section]
\newtheorem{corollary}[theorem]{Corollary}
\newtheorem{lemma}[theorem]{Lemma}

\newtheorem{lemma and definition}[theorem]{Lemma and Definition}
\newtheorem{proposition}[theorem]{Proposition}
\newtheorem{definition}[theorem]{Definition}
\newtheorem{exam}[theorem]{Example}

\newtheorem{remark}[theorem]{Remark}

\newtheorem{the construction}[theorem]{THE CONSTRUCTION}

\newcommand{\field}[1]{\mathbb{#1}}

\newcommand{\Z }{\field{Z}}
\newcommand{\N }{\field{N}}

\begin{document}
	
\title[Symmetric and Pseudo-Symmetric Numerical Semigroups]{Symmetric and Pseudo-Symmetric Numerical Semigroups via Young Diagrams and Their Semigroup Rings}
\author{Meral SÜER}
\address{Department of Mathematics, Faculty of Science and Letters, Batman University, Batman, Turkey}
\email{meral.suer@batman.edu.tr}
\author{Mehmet YEŞİL}
\address{Department of Mathematics, Faculty of Science and Letters, Batman University, Batman, Turkey}
\email[corresponding author]{mehmet-yesil@outlook.com}
\maketitle

\begin{abstract}
	This paper studies Young diagrams of symmetric and pseudo-symmetric numerical semigroups and describes new operations on Young diagrams as well as numerical semigroups. These provide new decompositions of symmetric and pseudo-symmetric semigroups into an over semigroup and its dual. It is also given exactly for what kind of numerical semigroup $S$, the semigroup ring $\Bbbk[\![S]\!]$ has at least one Gorenstein subring and has at least one Kunz subring.
\end{abstract}

\section{Introduction}

Throughout this paper, $\Z$ and $\N$ will denote the set of integers and the set of positive integers, respectively. We also set $\N_{0}=\N\cup\{0\}$. A numerical set $S$ is a subset of $\N_{0}$ containing $0$ and having finite complement in $\N_{0}$. It is clear that $\N_{0}$ is a numerical set with empty complement. A numerical set $S$ is a numerical semigroup, if $x,y \in S \implies x+y\in S$.

Numerical semigroups have many applications in ring theory and algebraic geometry via the valuations of one-dimensional local Noetherian domains. Under certain conditions, these rings can be characterized in terms of their value groups which are numerical semigroups. In particular, in \cite{Ku}, Kunz proves that one dimensional analytically irreducible Noetherian local rings are Gorenstein if and only if their value semigroups are symmetric. Also in \cite{BDF}, it is proved that a numerical semigroup is pseudo-symmetric if and only if its semigroup ring is a Kunz ring.

Numerical semigroups can be visualised with Young diagrams\footnote{In literature, Young diagrams are also called Young tableaux.}. A Young diagram\footnote{A Young diagram is also defined as a series of top-aligned columns of boxes such that the number of boxes in each column is not less than the number of boxes in the column immediately to the right of it.} is a series of left aligned rows of boxes such that the number of boxes in each row is not less than the number of boxes in the row immediately below it. Young diagrams are combinatorial objects that are useful in several branches of mathematics including but not limited to representation theory. The connection between numerical sets and Young diagrams is given in \cite{KN} and \cite{HBN}. As an important application of these objects, in \cite{N1}, Arf numerical semigroups are characterised via their Young diagrams.

In this paper, we study the Young diagrams of symmetric and pseudo-symmetric numerical semigroups. We define new binary operations on Young diagrams and also on numerical semigroups which allow us to characterize symmetric and pseudo-symmetric numerical semigroups by another numerical semigroup and its dual. This is actually a decomposition of the numerical semigroup under consideration into an over semigroup and its dual. These decompositions seem to deserve further investigation because of the ring theoretic correspondence of symmetric and pseudo-symmetric semigroups. Furthermore, we provide characterizations of which semigroup rings having at least one Gorenstein subring and having at least one Kunz subring.

This paper is organised as follows. In section \ref{sec2}, we gather necessary background of numerical sets and Young diagrams that we use in latter sections. In section \ref{sec3}, we introduce new binary operations on Young diagrams and give their numerical set correspondences. In section \ref{sec4}, we interpret the newly defined operations on symmetric and pseudo-symmetric numerical semigroups and we prove some technical lemmas. In particular, we prove the main results, Theorem \ref{tps2}, Theorem \ref{tps1}, and Theorem \ref{thm} of this paper using these operations on numerical semigroups and their duals. Finally, in section \ref{sec5}, we interpret the results of this paper in correspondence with ring theory, and we characterize semigroups whose semigroup rings have at least one Gorenstein subring and have at least one Kunz subring.

\section{Numerical Sets and Young Diagrams}\label{sec2}

A numerical set $S$ is said to be proper, if $S\neq \N_{0}$. Let $S$ be a proper numerical set. We denote the complement of $S$ in $\N_{0}$ by $G(S)$. The elements of $G(S)$ are called gaps of $S$. The number of gaps of $S$ is called genus of $S$, and denoted by $g(S)$. The largest gap of $S$ is called the Frobenius number of $S$, and denoted by $F(S)$. The number $F(S)+1$ is called the conductor of $S$, and denoted by $C(S)$. $C(S)$ is the smallest element of $S$ such that $n\in\N_{0}$ and $n>C(S) \implies n\in S$. Note that $F(\N_{0})=-1$ and $C(N_{0})=0$.

The elements of a proper numerical set $S$ that are smaller than $C(S)$ are called the small elements of $S$. If $S$ has $n$ small elements, we list them as $0=s_{0}<s_{1}<\dots<s_{n-1}$, and we write $$S=\{0,s_{1},\dots,s_{n-1},s_{n}=C(S),\rightarrow\},$$ where the arrow at the end means that all integers greater than $C(S)$ belong to $S$.

\begin{exam}\label{ex1}
$S=\{ 0,2,3,6,8,9,11,\rightarrow \}$ is a numerical set which has the complement $G(S)=\{ 1,4,5,7,10 \}$ and genus $g(s)=5$. Notice that the Frobenius number of $S$ is $F(S)=10$, and the conductor of $S$ is $C(S)=11$.
\end{exam}

Let $Y$ be a Young diagram with $n$ columns and $k$ rows. The number of boxes in a column (or a row) is called the length of that column (or that row). The hook of a box in $Y$ is the shape formed by the boxes directly to the right of it, the boxes directly below it, and the box itself. The number of boxes in the hook of a box is called the hook length of that box.

\begin{exam}\label{ex2}
Here is an example of a Young diagram $Y$ with $6$ columns and $5$ rows in which the boxes that contains bullets forms the hook of the box lying in the second column and second row.
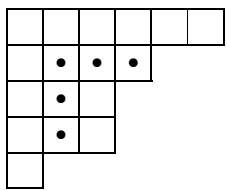
\vspace{1em}
\begin{center}
\tiny{\ytableaushort
{\none,\none\bullet\bullet\bullet, \none\bullet, \none\bullet}
* {6,4,3,3,1}}
\end{center}
\vspace{1em}
The hook length of that box is $5$.
\end{exam}

Let $S$ be a numerical set. We can construct a Young diagram $Y_{S}$ corresponding to $S$ by drawing a continuous polygonal path that starts from the origin in $\Z^{2}$. Starting with $s=0$,
\begin{enumerate}
	\item if $s \in S$, draw a line of unit length to the right,
	\item if $s \notin S$, draw a line of unit length to up,
\end{enumerate}
and repeat it for $s+1$. We continue this until $s=F(S)$. The lattice lying above this polygonal path and the horizontal line that is $g(S)$ units above the origin defines the corresponding Young diagram $Y_{S}$. It is clear that every Young diagram corresponds to a unique proper numerical set. Thus the correspondence $S\rightarrow Y_{S}$ is a bijection between the set of proper numerical sets and the set of Young diagrams. For example, the numerical set in Example \ref{ex1} corresponds to the Young diagram in Example \ref{ex2}.

Let $S=\{0,s_{1},\dots,s_{n-1},s_{n}=C(S),\rightarrow\}$ be a numerical set with corresponding Young diagram $Y_{S}$. By the construction, it is easy to see that $Y_{S}$ has $g(S)$ rows and $n$ columns. For each $i=0,1,\dots,n-1$, we identify the $i$th column of $Y_{S}$ with the set of hook lengths of boxes in it, denoted by $G_{i}(S)$, which corresponds to $s_{i}$. Also, the $i$th row from the bottom corresponds to the $i$th gap of $S$. The hook length of the box in the first column and $i$th row is the $i$th gap of $S$. Thus, $G_{0}(S)=G(S)$.

\begin{proposition} \label{prp} \cite[Section 2]{N1}
	Let $S=\{0,s_{1},\dots,s_{n-1},s_{n},\rightarrow\}$ be a numerical set with corresponding Young diagram $Y_{S}$. Then:
	\begin{enumerate}
		\item For each $i\in \{ 0,1,\dots,n-1 \}$, the hook length of the top box of the $i$th column of $Y_{S}$ is $F(S)-s_{i}$,
		\item $S$ is a numerical semigroup if and only if $G_{i}(S)\subseteq G_{0}(S)$ for each $i\in \{ 0,1,\dots,n-1 \}$.
	\end{enumerate}
\end{proposition}

\begin{exam}
Let $S=\{ 0,2,3,6,8,9,11,\rightarrow \}$. $S$ has $6$ small elements, the complement $G(S)=\{ 1,4,5,7,10 \}$ and the genus $g(S)=5$. Therefore, the corresponding Young diagram $Y_{S}$ has $6$ columns and $5$ rows, and the hook lengths of the boxes of first column and first row in $Y_{S}$ are shown in the picture below.
\vspace{1em}
\begin{center}
\tiny{\ytableaushort
{{10}87421,7,5,4,1}
*{6,4,3,3,1}}
\end{center}
\vspace{1em}
\end{exam} 

\section{Discrete, End-to-end, Conjoint and Overlap Sums}\label{sec3}
In this section, we define new operations on Young diagrams, and we give the numerical set correspondence of these operations. Henceforth, all numerical sets are proper.

\begin{definition}
	Let $Y$ be a Young diagram with $n$ columns and $k$ rows, $Z$ be a Young diagram with $m$ columns and $l$ rows. Glueing $Z$ above $Y$ as putting a row of boxes of length $n$ above $Y$ and then uniting the top right corner of this row and the bottom left corner of the first column of $Z$ is called the discrete sum of $Y$ and $Z$, denoted by $Y\boxplus_{D}Z$, which is a Young diagram with $n+m$ columns and $k+l+1$ rows.
\end{definition}

Next example illustrates the discrete sum of two Young diagrams.

\begin{exam}
	Let $Y$ and $Z$ be Young diagrams with $3$ columns and $3$ rows as below. Then $Y\boxplus_{D}Z$ is a Young diagram with $6$ columns and $7$ rows shown as below.
	\vspace{1em}
	\begin{center}
		\tiny{\begin{tabular}{ccccc}
				\ydiagram[*(yellow) ]{3,2,1}&$\boxplus_{D}$&\ydiagram[*(red) ]{3,1,1}&$=$&\ydiagram[*(yellow) \uparrow]
				{0,0,0,0,3,0,0}
				*[*(orange) \nearrow]
				{0,0,0,2+1,0,0}
				*[*(orange) \downarrow]
				{0,0,0,3,0,0,0}
				*[*(red) \swarrow]
				{0,0,3+1,0,0,0,0}
				*[*(red) ]{3+3,3+1,3+1}*[*(orange) ]{0,0,0,3}*[*(yellow) ]{0,0,0,0,3,2,1}*[*(white) ]{6,4,4,3,3,2,1}\\
				$\underbrace{\hspace{4.8em}}_{\displaystyle Y}$&&$\underbrace{\hspace{4.8em}}_{\displaystyle Z}$&&$\underbrace{\hspace{9.6em}}_{\displaystyle Y\boxplus_{D}Z}$
		\end{tabular}}
	\end{center}
\vspace{1em}
\end{exam}

\begin{definition}
	Let $Y$ be a Young diagram with $n$ columns and $k$ rows, $Z$ be a Young diagram with $m$ columns and $l$ rows. Glueing $Z$ above $Y$ as uniting the top right corner of the first row of $Y$ and the bottom left corner of the first column of $Z$ is called the end-to-end sum of $Y$ and $Z$, denoted by $Y\boxplus_{E}Z$, which is a Young diagram with $n+m$ columns and $k+l$ rows. 
\end{definition}

Next example illustrates the end-to-end sum of two Young diagrams.

\begin{exam}
	Let $Y$ and $Z$ be Young diagrams with $3$ columns and $3$ rows as below. Then $Y\boxplus_{E}Z$ is a Young diagram with $6$ columns and $6$ rows shown as below.
	
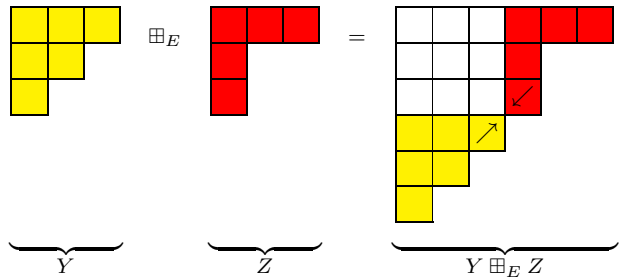
\vspace{1em}
	\begin{center}
		\tiny{\begin{tabular}{ccccc}
				\ydiagram[*(yellow) ]{3,2,1}&$\boxplus_{E}$&\ydiagram[*(red) ]{3,1,1}&$=$&\ydiagram[*(yellow) \nearrow]
				{0,0,0,2+1,0,0}
				*[*(red) \swarrow]
				{0,0,3+1,0,0,0}
				*[*(red) ]{3+3,3+1,3+1}*[*(yellow) ]{0,0,0,3,2,1}*[*(white) ]{6,4,4,3,2,1}\\
				$\underbrace{\hspace{4.8em}}_{\displaystyle Y}$&&$\underbrace{\hspace{4.8em}}_{\displaystyle Z}$&&$\underbrace{\hspace{9.6em}}_{\displaystyle Y\boxplus_{E}Z}$
		\end{tabular}}
	\end{center}
\vspace{1em}
\end{exam}

\begin{definition}
Let $Y$ be a Young diagram with $n$ columns and $k$ rows, $Z$ be a Young diagram with $m$ columns and $l$ rows. Glueing $Z$ above $Y$ as putting the first column of $Z$ on top of the last column $Y$ is called the conjoint sum of $Y$ and $Z$, denoted by $Y\boxplus_{C}Z$, which is a Young diagram with $n+m-1$ columns and $k+l$ rows.  
\end{definition}

Next example illustrates the conjoint sum of two Young diagrams.

\begin{exam}
Let $Y$ and $Z$ be Young diagrams with $3$ columns and $3$ rows as below. Then $Y\boxplus_{C}Z$ is a Young diagram with $5$ columns and $6$ rows shown as below.
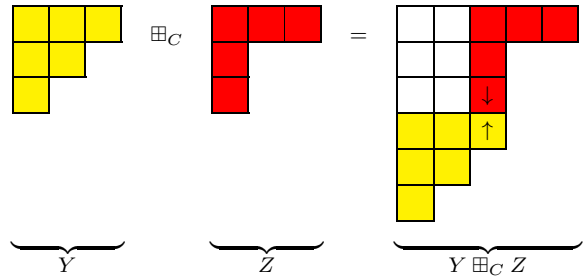
\vspace{1em}
\begin{center}
\tiny{\begin{tabular}{ccccc}
\ydiagram[*(yellow) ]{3,2,1}&$\boxplus_{C}$&\ydiagram[*(red) ]{3,1,1}&$=$&\ydiagram[*(yellow) \uparrow]
{0,0,0,2+1,0,0}
*[*(red) \downarrow]
{0,0,2+1,0,0,0}
*[*(red) ]{2+3,2+1,2+1}*[*(yellow) ]{0,0,0,3,2,1}*[*(white) ]{5,3,3,3,2,1}\\
$\underbrace{\hspace{4.8em}}_{\displaystyle Y}$&&$\underbrace{\hspace{4.8em}}_{\displaystyle Z}$&&$\underbrace{\hspace{8.0em}}_{\displaystyle Y\boxplus_{C}Z}$
\end{tabular}}
\end{center}
\vspace{1em}
\end{exam}

\begin{definition}
	Let $Y$ be a Young diagram with $n$ columns and $k$ rows, $Z$ be a Young diagram with $m$ columns and $l$ rows. Glueing $Z$ above $Y$ as overlapping the last box of first column of $Z$ and the last box of the first row of $Y$ is called the overlap sum of $Y$ and $Z$, denoted by $Y\boxplus_{O}Z$, which is a Young diagram with $n+m-1$ columns and $k+l-1$ rows. 
\end{definition}

Next example illustrates the overlap sum of two Young diagrams.

\begin{exam}
	Let $Y$ and $Z$ be Young diagrams with $3$ columns and $3$ rows as below. Then $Y\boxplus_{O}Z$ is a Young diagram with $5$ columns and $5$ rows shown as below.
	
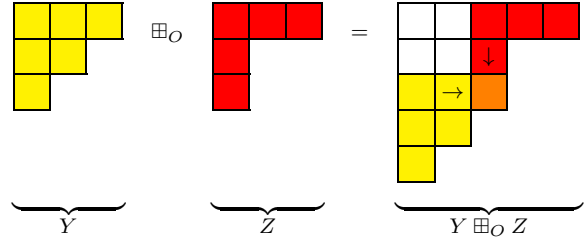
\vspace{1em}
	\begin{center}
		\tiny{\begin{tabular}{ccccc}
				\ydiagram[*(yellow) ]{3,2,1}&$\boxplus_{O}$&\ydiagram[*(red) ]{3,1,1}&$=$&\ydiagram[*(yellow) \rightarrow]
				{0,0,1+1,0,0}
				*[*(red) \downarrow]
				{0,2+1,0,0,0}
				*[*(red) ]{2+3,2+1}*[*(yellow) ]{0,0,2,2,1}*[*(orange) ]{0,0,2+1,0,0}*[*(white) ]{5,3,3,2,1}\\
				$\underbrace{\hspace{4.8em}}_{\displaystyle Y}$&&$\underbrace{\hspace{4.8em}}_{\displaystyle Z}$&&$\underbrace{\hspace{8.0em}}_{\displaystyle Y\boxplus_{O}Z}$
		\end{tabular}}
	\end{center}
\vspace{1em}
\end{exam}

Because of the bijection between the set of Young diagrams and the set of proper numerical sets, one can define discrete, end-to-end, conjoint and overlap sums for numerical sets as well.

If we are given two numerical sets $S$ and $T$ with corresponding Young diagrams $Y_{S}$ and $Y_{Z}$, respectively, we think of discrete, end-to-end, conjoint and overlap sums of $S$ and $T$ as the corresponding numerical sets to discrete, end-to-end, conjoint and overlap sums of $Y_{S}$ and $Y_{Z}$, respectively. 

If we add $Y_{Z}$ to $Y_{S}$ conjointly, we actually lift $Y_{Z}$ the number of rows of $Y_{S}$ times up and move it the number of columns of $Y_{S}$ minus one times to the right. With this process we move $Y_{Z}$ exactly the number of the hook length of the top left box in $Y_{S}$ times. Now if we think of the new polygonal path that have been constructed, we have made the conductor of $S$ and $0$ in $T$ vanished, and we have added $F(S)$ to the non-zero elements of $T$. Therefore, the small elements of the corresponding numerical set to the young diagram of $Y_{S}\boxplus_{C}Y_{T}$ consist of the small elements of $S$ less than the conductor of $S$ and the non-zero small elements of $T$ added $F(S)$. Also the conductor of $Y_{S}\boxplus_{C}Y_{T}$ is the conductor of $T$ added $F(S)$. In similar ways, we may talk about the construction of discrete, end-to-end and overlap sums of $S$ and $T$.

\begin{definition}
	Let $S=\{0,s_{1},\dots,s_{n},\rightarrow\}$ and $T=\{0,t_{1},\dots,t_{m},\rightarrow\}$ be numerical sets. The discrete sum of $S$ and $T$ is the numerical set 
	$$S\boxplus_{D}T=\{0,s_{1},\dots,s_{n-1},s_{n}+1,t_{1}+s_{n}+1,\dots,t_{m}+s_{n}+1,\rightarrow\},$$
	the end-to-end sum of $S$ and $T$ is the numerical set 
	$$S\boxplus_{E}T=\{0,s_{1},\dots,s_{n},t_{1}+s_{n},\dots,t_{m}+s_{n},\rightarrow\},$$
	the conjoint sum of $S$ and $T$ is the numerical set $$S\boxplus_{C}T=\{0,s_{1},\dots,s_{n-1},t_{1}+s_{n}-1,\dots,t_{m}+s_{n}-1,\rightarrow\},$$
	and the overlap sum of $S$ and $T$ is the numerical set 
	$$S\boxplus_{O}T=\{0,s_{1},\dots,s_{n-1},t_{1}+s_{n}-2,\dots,t_{m}+s_{n}-2,\rightarrow\}.$$
\end{definition}

Next lemma gives the most important properties of newly defined sums of numerical sets whose proofs are just by definitions.

\begin{lemma} \label{l2}
	Let $S=\{0,s_{1},\dots,s_{n},\rightarrow\}$ and $T=\{0,t_{1},\dots,t_{m},\rightarrow\}$ be numerical sets with $G(S)=\{ a_{1},\dots,a_{k} \}$ and $G(T)=\{ b_{1},\dots,b_{l} \}$, respectively. Then
	\begin{enumerate}
		\item $G(S\boxplus_{D}T)=\{ a_{1},\dots,a_{k}, a_{k}+1,a_{k}+b_{1}+2,\dots,a_{k}+b_{l}+2 \}$ and $F(S\boxplus_{D}T)=a_{k}+b_{l}+2$,
		\item $G(S\boxplus_{E}T)=\{ a_{1},\dots,a_{k},a_{k}+b_{1}+1,\dots,a_{k}+b_{l}+1 \}$ and $F(S\boxplus_{E}T)=a_{k}+b_{l}+1$,
		\item $G(S\boxplus_{C}T)=\{ a_{1},\dots,a_{k},a_{k}+b_{1},\dots,a_{k}+b_{l} \}$ and $F(S\boxplus_{C}T)=a_{k}+b_{l}$,
		\item $G(S\boxplus_{O}T)=\{ a_{1},\dots,a_{k-1},a_{k}+b_{1}-1,\dots,a_{k}+b_{l}-1 \}$ and $F(S\boxplus_{O}T)=a_{k}+b_{l}-1$.
	\end{enumerate} 
\end{lemma}

Notice that these sums of numerical sets are closed non-commutative binary operations on the set of proper numerical sets. Furthermore, it is easy to see that $\{0,2,\rightarrow\}$ is the identity element for the overlap sum of numerical sets. However, next proposition shows that conjoint and discrete sums are not closed on the set of proper numerical semigroups. 

\begin{proposition} \label{p1}
	Let $S=\{0,s_{1},\dots,s_{n},\rightarrow\}$ be a numerical semigroup with $G(S)=\{ a_{1},\dots,a_{k} \}$ where $a_{k}=F(S)$. If $s_{n}$ is not a minimal generator for $S$, then $S\boxplus_{C}S$ and $S\boxplus_{D}S$ are only numerical sets. 
\end{proposition}

\begin{proof}
Suppose that $s_{n}$ is not a minimal generator for $S$, i.e. $s_{n}=s_{i}+s_{j}$ for some $i,j\in\{ 1,2,\dots,n-1\}$. Since $s_{n}=a_{k}+1\in G(S\boxplus_{D}S)$, the proof is straightforward for $S\boxplus_{D}S$. Now suppose also the contrary that $S\boxplus_{C}S=\{ 0,s_{1},\dots,s_{n-1},s_{1}+s_{n}-1,\dots,2s_{n}-1 \}$ is a numerical semigroup. Then $s_{n}=s_{i}+s_{j}\in S\boxplus_{C}S \implies s_{n}=s_{1}+s_{n}-1 \implies s_{1}=1$. However, this is a contradiction since $S$ is proper.
\end{proof}
\vspace{1em}

\section{Young Diagrams of  Symmetric and Pseudo-symmetric Numerical Semigruops}\label{sec4}

In this section, we investigate the behaviour of  end-to-end and overlap sums on symmetric numerical semigroups, and the behaviour of discrete and conjoint sums on pseudo-symmetric semigroups.

Recall that a numerical semigroup $S$ is symmetric  if $F(S)$ is odd and $x\in \Z\setminus S \implies F(S)-x\in S$. Also a numerical semigorup $S$ is pseudo-symmetric if $F(S)$ is even and $x\in\Z\setminus S \implies x=\frac{F(S)}{2}$ or $F(S)-x\in S$.

\begin{proposition}\cite[Corollary 4.5]{RG} \label{pps1} Let $S$ be a numerical semigroup.
	\begin{enumerate}
		\item $S$ is symmetric if and only if $g(S)=\dfrac{F(S)+1}{2}$,
		\item $S$ is pseudo-symmetric if and only if $g(S)=\dfrac{F(S)+2}{2}$.
	\end{enumerate}
\end{proposition}

\begin{definition}
	For a numerical semigroup $S$ with $G(S)=\{ a_{1},\dots,a_{k} \}$ and the Frobenius number $F(S)=a_{k}$, we define the dual of $S$ as $$S^{*}=\{0,F(S)-a_{k-1},F(S)-a_{k-2},\dots,F(S)-a_{1},s_{n},\rightarrow \}$$ whose Frobenius number is again $F(S)$.
\end{definition}

Notice that the corresponding Young diagram $Y_{S^{*}}$ of $S^{*}$ is a Young diagram that we get from interchanging the rows and columns of the corresponding Young diagram $Y_{S}$ of $S$. For more information, we refer to \cite{HBN} and \cite{KN}.

\begin{lemma} \label{l3}
	Let $S=\{ 0,s_{1},\dots,s_{n} \}$ be a numerical semigroup. Then
	\begin{enumerate}
		\item $g(S)=s_{n}-n$ and $g(S^{*})=n$,
		\item $g(S\boxplus_{E}S^{*})=s_{n}$ and $g(S\boxplus_{O}S^{*})=s_{n}-1$,
		\item $g(S\boxplus_{D}S^{*})=s_{n}+1$ and $g(S\boxplus_{C}S^{*})=s_{n}$.
	\end{enumerate}
\end{lemma}

\begin{proof}
	The first is straightforward by definitions. For the others, we have $g(S\boxplus_{E}S^{*})=g(S)+g(S^{*})$ and $g(S\boxplus_{O}S^{*})=g(S)+g(S^{*})-1$,  $g(S\boxplus_{D}S^{*})=g(S)+g(S^{*})+1$ and $g(S\boxplus_{C}S^{*})=g(S)+g(S^{*})$ by Lemma \ref{l2}, and so we are done.
\end{proof}

\begin{remark} \label{s2}
	If $S=\{0,s_{1},\dots,s_{n},\rightarrow\}$ is a symmetric numerical semigroup, then by definition, the of gaps of $S$ is $$G(S)=\{ a_{1},\dots,a_{n-1}, F(S) \}$$ such that $S=\{ 0, F(S)-a_{n-1},\dots,F(S)-a_{1},s_{n},\rightarrow \}$ where $s_{i}=F(S)-a_{n-i}$ for $i=1,2,\dots,n-1$. Thus, $S=S^{*}$.
\end{remark}

Next theorem gives a decomposition of a symmetric numerical semigroup into an over semigroup and its dual.

\begin{theorem} \label{tps2}
		For every symmetric numerical semigroup $S$, there exist a unique numerical semigorup $T$ such that $S=T\boxplus_{E}T^{*}$ or $S=T\boxplus_{O}T^{*}$.
\end{theorem}

\begin{proof}
By Remark \ref{s2}, $S=\{ 0,s_{1},\dots,s_{k},s_{k+1},\dots,s_{n},\rightarrow \}$ with $s_{n}=C(S), s_{k}\leq\frac{C(S)}{2}<s_{k+1}$, $G(S)=\{ a_{1},\dots,a_{n} \}$, $F(S)=a_{n}$ and $$S=S^{*}=\{ 0,F(S)-a_{n-1},\dots,F(S)-a_{1},s_{n}=C(S),\rightarrow \}.$$ We continue in two cases. First, if $\frac{C(S)}{2}\in S$, then we choose $T=\{ 0,s_{1},\dots,s_{k}=\frac{C(S)}{2},\rightarrow \}$ which is obviously a numerical semigroup with $G(T)=\{ a_{1},\dots,a_{n-k} \}$ and $F(T)=\frac{C(S)}{2}-1$. Then $$T^{*}=\{0, F(T)-a_{n-k-1},\dots,F(T)-a_{1},\frac{C(S)}{2},\rightarrow \}.$$ Therefore, by definition of end-to-end sum
\begin{align*}
	T\boxplus_{E}T^{*}=&\{ 0,s_{1},\dots,s_{k-1},\frac{C(S)}{2},\frac{C(S)}{2}+F(T)-a_{n-k-1},\dots\\
	&\hspace{4cm}\dots,\frac{C(S)}{2}+F(T)-a_{1},C(S), \rightarrow \}\\
	=&\{ 0,s_{1},\dots,s_{k},C(S)-1-a_{n-k-1},\dots,C(S)-1-a_{1},C(S), \rightarrow \}\\
	=&\{ 0,s_{1},\dots,s_{k},F(S)-a_{n-k-1},\dots,F(S)-a_{1},C(S), \rightarrow \}\\
	=&\{ 0,s_{1},\dots,s_{k},s_{k+1},\dots,s_{n-1},s_{n}, \rightarrow \}=S
\end{align*}
Now if $\frac{C(S)}{2}\notin S$, then $F(S)-\frac{C(S)}{2}\in S \implies \frac{C(S)}{2}-1\in S\implies s_{k}=\frac{C(S)}{2}-1$. Thus, we choose $T=\{ 0,s_{1},\dots,s_{k},\frac{C(S)}{2}+1,\rightarrow \}$ which is obviously a numerical semigroup with $G(T)=\{  a_{1},\dots,a_{n-k-1},\frac{C(S)}{2} \}$ and $F(T)=\frac{C(S)}{2}$. Then $$T^{*}=\{0, F(T)-a_{n-k-1},\dots,F(T)-a_{1},\frac{C(S)}{2}+1,\rightarrow \}.$$ Therefore, by definition of overlap sum
\begin{align*}
	T\boxplus_{O}T^{*}=&\{ 0,s_{1},\dots,s_{k},\frac{C(S)}{2}-1+F(T)-a_{n-k-1},\dots\\
	&\hspace{4cm}\dots,\frac{C(S)}{2}-1+F(T)-a_{1},C(S), \rightarrow \}\\
	=&\{ 0,s_{1},\dots,s_{k},C(S)-1-a_{n-k-1},\dots,C(S)-1-a_{1},C(S), \rightarrow \}\\
	=&\{ 0,s_{1},\dots,s_{k},F(S)-a_{n-k-1},\dots,F(S)-a_{1},C(S), \rightarrow \}\\
	=&\{ 0,s_{1},\dots,s_{k},s_{k+1},\dots,s_{n-1},s_{n}, \rightarrow \}=S
\end{align*}
\end{proof}

\begin{exam}
	Let $S=\{ 0,4,7,8,10,11,12,14,\rightarrow \}$ with $F(S)=14$ and $G(S)=\{ 1,2,3,5,6,9,13 \}$, which is a  pseudo-symmetric numerical semigroup containing $\frac{C(S)}{2}=7$. Therefore, by Theorem \ref{tps2}, $T=\{ 0,4,7,\rightarrow \}$ with $G(T)=\{ 1,2,3,5,6 \}$ and $T^{*}=\{ 0,1,3,4,5,7,\rightarrow \}$ where the end-to-end sum of $T$ and $T^{*}$ is $S$, i.e. $T\boxplus_{E}T^{*}=S$. The Young diagram of $Y_{S}=Y_{T}\boxplus_{E}Y_{T^{*}}$ is shown as below.
	
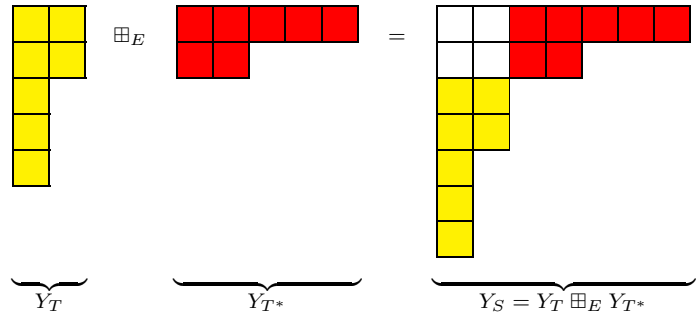
\vspace{1em}
	\begin{center}
		\tiny{\begin{tabular}{ccccc}
				\ydiagram[*(yellow) ]{2,2,1,1,1}&$\boxplus_{E}$&\ydiagram[*(red) ]{5,2}&$=$&\ydiagram[*(red) ]{2+5,2+2}*[*(yellow) ]{,0,0,2,2,1,1,1}*[*(white) ]{7,4,2,2,1,1,1}\\
				$\underbrace{\hspace{3.2em}}_{\displaystyle Y_{T}}$&&$\underbrace{\hspace{8.0em}}_{\displaystyle Y_{T^{*}}}$&&$\underbrace{\hspace{11.2em}}_{\displaystyle Y_{S}=Y_{T}\boxplus_{E}Y_{T^{*}}}$
		\end{tabular}}
	\end{center}
\vspace{1em}
\end{exam}

\begin{exam}
Let $$S=\{ 0,5,8,10,13,15,16,18,20,21,23,24,25,26,28,\rightarrow \}$$ with $G(S)=\{ 1,2,3,4,6,7,9,11,12,14,17,19,22,27 \}$ and $F(S)=27$, which is a symmetric numerical semigroup not containing $\frac{C(S)}{2}=14$. Therefore, by Theorem \ref{tps2}, $T=\{ 0,5,8,10,13,15,\rightarrow \}$ with $G(T)=\{ 1,2,3,4,6,7,9,11,12,14 \}$ and $$T^{*}=\{ 0,2,3,5,7,8,10,11,12,13,15,\rightarrow \}$$ where the overlap sum of $T$ and $T^{*}$ is $S$, i.e. $T\boxplus_{O}T^{*}=S$. The Young diagram of $Y_{S}=Y_{T}\boxplus_{O}Y_{T^{*}}$ is shown as below.
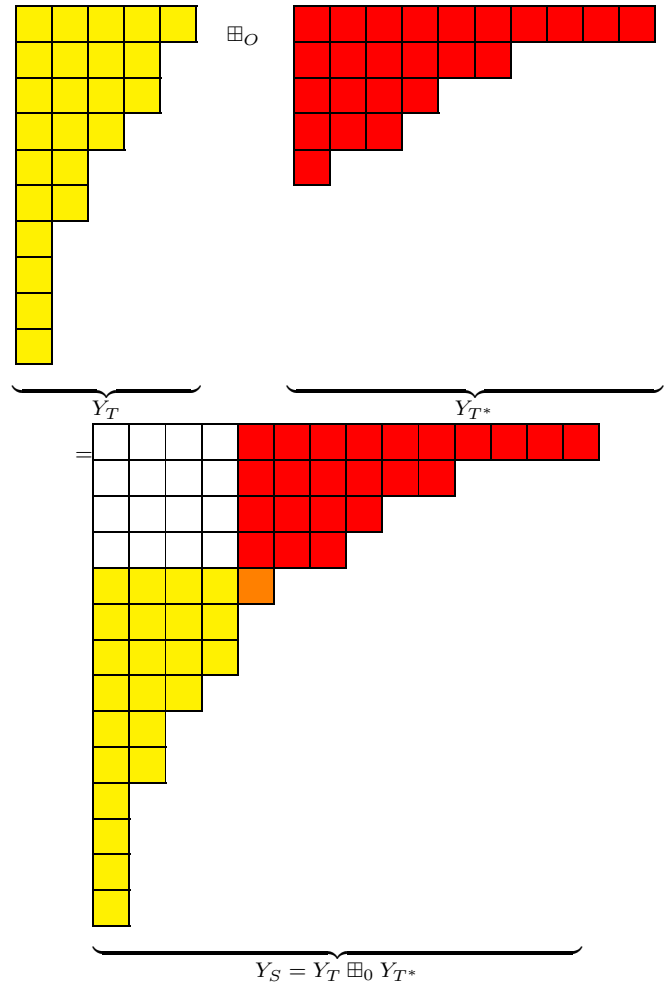
\vspace{1em}
\begin{center}
	\tiny{\begin{tabular}{ccccc}
			\ydiagram[*(yellow) ]{5,4,4,3,2,2,1,1,1,1}&$\boxplus_{O}$&\ydiagram[*(red) ]{10,6,4,3,1}\\
			$\underbrace{\hspace{8em}}_{\displaystyle Y_{T}}$&&$\underbrace{\hspace{16em}}_{\displaystyle Y_{T^{*}}}$
	\end{tabular}

\begin{tabular}{ccccc}
	$=$\ydiagram[*(red) ]{4+10,4+6,4+4,4+3}*[*(yellow) ]{0,0,0,0,4,4,4,3,2,2,1,1,1,1}*[*(orange) ]{0,0,0,0,4+1}*[*(white) ]{14,10,8,7,5,4,4,3,2,2,1,1,1,1}\\
	$\underbrace{\hspace{20.8em}}_{\displaystyle Y_{S}=Y_{T}\boxplus_{0}Y_{T^{*}}}$
\end{tabular}}
\end{center}
\vspace{1em}
\end{exam}

\begin{lemma} \label{ls2}
	If $S=\{0,s_{1},\dots,s_{n},\rightarrow\}$ is a symmetric numerical semigroup not equal to $\{ 0,2,\rightarrow \}$, then $s_{n}$ is always not a minimal generator for $S$. 
\end{lemma}

\begin{proof}
	Assume that $S$ is a symmetric numerical semigroup and $S\neq\{ 0,2,\rightarrow \}$. Then for each small element $s_{i}$, we have $s_{i}=F(S)-a_{i}$ such that $a_{i}\in G(S)$ and $i=1,\dots,n-1$. Moreover, at least one of $a_{i}$, we should have $a_{i}+1 \in S$. Hence, for that specific $a_{i}$, $s_{i}=F(S)-a_{i} \implies s_{i}=s_{n}-1-a_{i} \implies s_{i}+a_{i}+1=s_{n}$.
\end{proof}

\begin{corollary} \label{cps1}
If $S=\{0,s_{1},\dots,s_{n},\rightarrow\}$ is a symmetric numerical semigroup not equal to $\{ 0,2,\rightarrow \}$, then $S\boxplus_{C}S^{*}$ and $S\boxplus_{D}S^{*}$ are only numerical sets.
\end{corollary}

\begin{proof}
For a symmetric $S$, we have $S=S^{*}$. Then by Proposition \ref{p1} and Lemma \ref{ls2}, we are done.
\end{proof}

\begin{remark} \label{rps1}
If $S=\{0,s_{1},\dots,s_{n},\rightarrow\}$ is a pseudo-symmetric numerical semigroup, then by definition, the set of gaps of $S$ is $$G(S)=\{ a_{1},\dots,a_{n-k-1},\frac{F(S)}{2},a_{n-k},\dots,a_{n-1},F(S) \}$$ such that $$a_{1}<\dots<a_{n-k-1}<\frac{F(S)}{2}<a_{n-k}<\dots<a_{n-1}< F(S)$$ and $S=\{ 0, F(S)-a_{n-1},\dots,F(S)-a_{1},s_{n},\rightarrow \}$ where $s_{i}=F(S)-a_{n-i}$ for $i=1,2,\dots,n-1$.
\end{remark}

Next theorem gives a decomposition of a pseudo-symmetric numerical semigroup into an over semigroup and its dual.

\begin{theorem} \label{tps1}
	For every pseudo-symmetric numerical semigroup $S$ not equal to $\{ 0,3,\rightarrow \}$, there exist a unique non-symmetric numerical semigorup $T$ such that $S=T\boxplus_{C}T^{*}$ or $S=T\boxplus_{D}T^{*}$.
\end{theorem}

\begin{proof}
	One can easily check that $\{ 0,3,\rightarrow \}=\{ 0,2,\rightarrow \}\boxplus_{D}\{ 0,2,\rightarrow \}$ and $\{ 0,2,\rightarrow \}$ is symmetric. Suppose now that $S\neq\{ 0,3,\rightarrow \}$.
	By Remark \ref{rps1}, $S=\{ 0,s_{1},\dots,s_{k},s_{k+1},\dots,s_{n},\rightarrow \}$ with $s_{k}<\dfrac{F(S)}{2}<s_{k+1}$, $$G(S)=\{ a_{1},\dots,a_{n-k-1},\frac{F(S)}{2},a_{n-k},\dots,a_{n-2},a_{n-1}, F(S) \},$$ and $S=\{ 0, F(S)-a_{n-1},\dots,F(S)-a_{1},s_{n},\rightarrow \}$ where $s_{i}=F(S)-a_{n-i}$ for $i=1,2,\dots,n-1$. We now continue in two cases. First, if $\frac{F(S)}{2}+1 \notin S$, we choose $T=\{ 0,s_{1},\dots,s_{k},\frac{F(S)}{2}+1,\rightarrow \}$ which is obviously a numerical semigroup with $G(T)=\{ a_{1},\dots,a_{n-k-1},\frac{F(S)}{2} \}$ and $F(T)=\frac{F(S)}{2}$. Then $$T^{*}=\{ 0,F(T)-a_{n-k-1},\dots,F(T)-a_{1},F(T)+1,\rightarrow \}.$$ Therefore, by definition of conjoint sum we have
	\begin{align*}
	 T\boxplus_{C}T^{*}&=\{ 0,s_{1},\dots,s_{k},2F(T)-a_{n-k-1},\dots,2F(T)-a_{1},2F(T)+1, \rightarrow \}\\
	 &=\{ 0,s_{1},\dots,s_{k},F(S)-a_{n-k-1},\dots,F(S)-a_{1},F(S)+1, \rightarrow \}\\
	 &=\{ 0,s_{1},\dots,s_{k},s_{k+1},\dots,s_{n-1},s_{n}, \rightarrow \}=S
	 \end{align*}
 Now, if $\frac{F(S)}{2}+1 \in S$, then $s_{k+1}=\frac{F(S)}{2}+1$. In this case, we choose $T=\{ 0,s_{1},\dots,s_{k},\frac{F(S)}{2},\rightarrow \}$ which is obviously a numerical semigroup with $G(T)=\{ a_{1},\dots,a_{n-k-1} \}$ and $F(T)=\frac{F(S)}{2}-1=a_{n-k-1}$. Then $$T^{*}=\{ 0,F(T)-a_{n-k-2},\dots,F(T)-a_{1},F(T)+1,\rightarrow \}.$$ Therefore, by definition of discrete sum we have
 \begin{align*}
 	T\boxplus_{D}T^{*}=&\{ 0,s_{1},\dots,s_{k},F(T)+2,2F(T)-a_{n-k-2}+2,\dots\\&\hspace{4cm} \dots,2F(T)-a_{1}+2,2F(T)+3, \rightarrow \}\\
 	=&\{ 0,s_{1},\dots,s_{k},\frac{F(S)}{2}+1,F(S)-a_{n-k-2},\dots\\
 	&\hspace{5cm}\dots,F(S)-a_{1},F(S)+1, \rightarrow \}\\
 	=&\{ 0,s_{1},\dots,s_{k},s_{k+1},\dots,s_{n-1},s_{n}, \rightarrow \}=S
 \end{align*}
By Corollary \ref{cps1}, $T\boxplus_{C}T^{*}$ and $T\boxplus_{D}T^{*}$ are not numerical semigroups for a symmetric $T$. Hence, $T$ is not symmetric in any case.
\end{proof}

\begin{exam}
Let $S=\{ 0,6,7,11,12,13,14,15,17,\rightarrow \}$ with $G(S)=\{ 1,2,3,4,5,8,9,10,16 \}$ and $F(S)=16$, which is a  pseudo-symmetric numerical semigroup not containing $\frac{F(S)}{2}+1=9$. Therefore, by Theorem \ref{tps1}, $T=\{ 0,6,7,9,\rightarrow \}$ with $G(T)=\{ 1,2,3,4,5,8 \}$ and $T^{*}=\{ 0,3,4,5,6,7,9,\rightarrow \}$ where the conjoint sum of $T$ and $T^{*}$ is $S$, i.e. $T\boxplus_{C}T^{*}=S$. The Young diagram of $Y_{S}=Y_{T}\boxplus_{C}Y_{T^{*}}$ is shown as below.
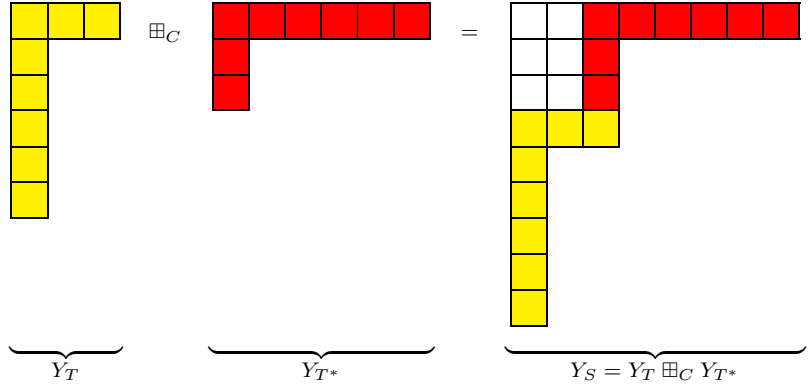
\vspace{1em}
\begin{center}
	\tiny{\begin{tabular}{ccccc}
		\ydiagram[*(yellow) ]{3,1,1,1,1,1}&$\boxplus_{C}$&\ydiagram[*(red) ]{6,1,1}&$=$&\ydiagram[*(red) ]{2+6,2+1,2+1}*[*(yellow) ]{0,0,0,3,1,1,1,1,1}*[*(white) ]{8,3,3,3,1,1,1,1,1}\\
		$\underbrace{\hspace{4.8em}}_{\displaystyle Y_{T}}$&&$\underbrace{\hspace{9.6em}}_{\displaystyle Y_{T^{*}}}$&&$\underbrace{\hspace{12.8em}}_{\displaystyle Y_{S}=Y_{T}\boxplus_{C}Y_{T^{*}}}$
	\end{tabular}}
\end{center}
\vspace{1em}
\end{exam}

\begin{exam}
Let $S=\{ 0,4,8,9,11,12,13,15,\rightarrow \}$ with $F(S)=14$ and $G(S)=\{ 1,2,3,5,6,7,10,14 \}$, which is a  pseudo-symmetric numerical semigroup containing $\frac{F(S)}{2}+1=8$. Therefore, by Theorem \ref{tps1}, $T=\{ 0,4,7,\rightarrow \}$ with $G(T)=\{ 1,2,3,5,6 \}$ and $T^{*}=\{ 0,1,3,4,5,7,\rightarrow \}$ where the discrete sum of $T$ and $T^{*}$ is $S$, i.e. $T\boxplus_{D}T^{*}=S$. The Young diagram of $Y_{S}=Y_{T}\boxplus_{D}Y_{T^{*}}$ is shown as below.
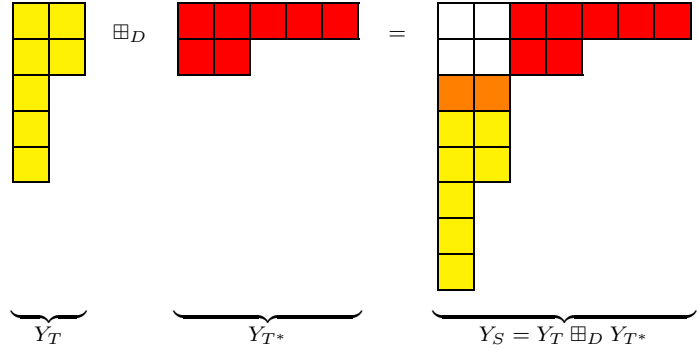
\vspace{1em}
\begin{center}
	\tiny{\begin{tabular}{ccccc}
		\ydiagram[*(yellow) ]{2,2,1,1,1}&$\boxplus_{D}$&\ydiagram[*(red) ]{5,2}&$=$&\ydiagram[*(red) ]{2+5,2+2}*[*(orange) ]{0,0,2}*[*(yellow) ]{0,0,0,2,2,1,1,1}*[*(white) ]{7,4,2,2,1,1,1}\\
		$\underbrace{\hspace{3.2em}}_{\displaystyle Y_{T}}$&&$\underbrace{\hspace{8.0em}}_{\displaystyle Y_{T^{*}}}$&&$\underbrace{\hspace{11.2em}}_{\displaystyle Y_{S}=Y_{T}\boxplus_{D}Y_{T^{*}}}$
	\end{tabular}}
\end{center}
\vspace{1em}
\end{exam}

For a numerical semigroup $S$, next example shows that $S\boxplus_{E}S^{*}$ or $S\boxplus_{O}S^{*}$ is not always a symmetric numerical semigroup, and that $S\boxplus_{C}S^{*}$ or $S\boxplus_{D}S^{*}$ is not always a pseudo-symmetric numerical semigroup.

\begin{exam} \label{exam}
	Let $S=\{ 0,3,5,6,8,\rightarrow \}$ be a numerical semigroup with $G(S)=\{ 1,2,4,7 \}$ and $S^{*}=S$. Then
	$$S\boxplus_{D}S^{*}=\{ 0,3,5,6,9,12,14,15,17,\rightarrow \},$$
	$$S\boxplus_{E}S^{*}=\{ 0,3,5,6,8,11,13,14,16,\rightarrow \},$$
	$$S\boxplus_{C}S^{*}=\{ 0,3,5,6,10,12,13,15,\rightarrow \},$$ and
	$S\boxplus_{O}S^{*}=\{ 0,3,5,6,9,11,12,14,\rightarrow \}$ are not numerical semigroups.
\end{exam}

However, we are able to determine exactly for what kind of numerical semigroup $S$ we get $S\boxplus_{E}S^{*}$ or $S\boxplus_{O}S^{*}$ a symmetric numerical semigroup, and $S\boxplus_{C}S^{*}$ or $S\boxplus_{D}S^{*}$ a pseudo-symmetric numerical semigroup.

\begin{theorem} \label{thm}
	Let $S=\{0,s_{1},\dots,s_{n},\rightarrow\}$ be numerical semigroup with $G(S)=\{ a_{1},\dots,a_{m} \}$. Then
	\begin{enumerate}
	\item $S\boxplus_{D}S^{*}$ is a numerical semigroup if and only if $s_{n}$ is a minimal generator of $S$, and $2s_{n}-s_{i}-s_{j}\neq s_{k}$ $\forall i,j,k \in \{0,\dots,n-1\}$.
	\item $S\boxplus_{E}S^{*}$ is a numerical semigroup if and only if $2s_{n}-s_{i}-s_{j}-1\neq s_{k}$ $\forall i,j,k \in \{0,\dots,n-1\}$.
	\item $S\boxplus_{C}S^{*}$ is a numerical semigroup if and only if $s_{n}$ is a minimal generator of $S$, and $2a_{m}-s_{i}-s_{j}\neq s_{k}$ $\forall i,j,k \in \{0,\dots,n-1\}$.
	\item $S\boxplus_{O}S^{*}$ is a numerical semigroup if and only if $2a_{m}-s_{i}-s_{j}-1\neq s_{k}$ $\forall i,j,k \in \{0,\dots,n-1\}$.
	\end{enumerate}
\end{theorem}

\begin{proof}
To prove the theorem, we will use Young diagram properties of $Y_{S}$ and $Y_{S^{*}}$. We will also use Proposition \ref{prp} repeatedly. Note that we have $S^{*}=\{ 0,a_{m}-a_{m-1},\dots,a_{m}-a_{1},s_{n},\rightarrow \}$.

Proof of (1): By definition, $S\boxplus_{D}S^{*}=\{ s_{1},\dots,s_{n-1},s_{n}+1,s_{n}+1+a_{m}-a_{m-1},\dots,s_{n}+1+a_{m}-a_{1},2s_{n}+1,\rightarrow \}$.
\begin{center}
	{\tiny
		\ytableausetup
		{mathmode, boxsize=5em}
		\begin{tabular}{ccccc}
			\begin{ytableau}
				m_{11} & \dots & m_{1j} & \dots & m_{1n}\\
				\vdots & \dots & \vdots & \dots & \vdots\\
				m_{i1} & \dots & m_{ij} & \dots & m_{in}\\
				\vdots & \dots & \vdots & \dots & \vdots\\
				m_{n1} & \dots & m_{nj} & \dots & m_{nn}\\
			\end{ytableau}&
			\begin{ytableau}
				a_{k}&\none[\dots]\\ \vdots&\none[\dots]\\ a_{m}- s_{i-1}&\none[\dots]\\ \vdots&\none[\dots]\\ a_{m}-s_{n-1}&\none[\dots]
			\end{ytableau}\\
			
			\begin{ytableau}
				s_{n}  & \dots & s_{n}-s_{j-1} & \dots & s_{n}-s_{n-1}
			\end{ytableau}\\
			\begin{ytableau}
				a_{k} & \dots & a_{k}- s_{j-1} & \dots & a_{k}-s_{n-1}\\
				\none[\vdots]& \none[\vdots]& \none[\vdots]&\none[\vdots]&\none[\vdots]
			\end{ytableau}
	\end{tabular}}	
\end{center}
The picture above illustrates the Young diagram of $S\boxplus_{D}S^{*}$. The right side of it continues with the Young diagram of $S^{*}$, and the bottom side continues with the Young diagram of $S$. We label the hook lengths of boxes in top left $n\times n$ matrix shaped part of $Y_{S\boxplus_{D}S^{*}}$ with $m_{ij}$. By proposition \ref{prp}, $m_{1j}=2s_{n}-s_{j-1}$ where $j=1,\dots n$. Let $c_{1},\dots,c_{n}$ be the lengths of first $n$ columns of $Y_{S\boxplus_{D}S^{*}}$, and let $r_{1},\dots,r_{n}$ be the lengths of first $n$ rows of $Y_{S\boxplus_{D}S^{*}}$. It is easy to see that $r_{i}+1=c_{i}$ for each $i=1,\dots,n$. Since $r_{i}-r_{i+1}=c_{i}-c_{i+1}$, we have $m_{i1}=2s_{n}-s_{i-1}$ where $i=1,\dots,n$. This means that $$G(S\boxplus_{D}S^{*})=\{ a_{1},\dots,a_{k},s_{n},2s_{n}-s_{n-1},\dots,2s_{n}-s_{1},2s_{n}\}.$$ Now, if we delete the first row of $Y_{S\boxplus_{D}S^{*}}$, we get another Young diagram representing a numerical set whose gaps are $a_{1},\dots,a_{k},s_{n},2s_{n}-s_{n-1},\dots,2s_{n}-s_{1}$. Thus, $m_{2j}=2s_{n}-s_{1}-s_{j-1}$ where $j=1,\dots,n$. If we continue in a similar way, we get all $m_{ij}=2s_{n}-s_{i-1}-s_{j-1}$.

By proposition \ref{prp} again, $S\boxplus_{D}S^{*}$ is a numerical semigroup if and only if hook lengths of all boxes in $Y_{S\boxplus_{D}S^{*}}$ are in $G(S\boxplus_{D}S^{*})$. Since $S$ is a numerical semigroup, hook lengths of boxes after $n$th columns and after $(n+1)$th rows are in $G(S)\subseteq G(S\boxplus_{D}S^{*})$. Therefore, we only need to check boxes in $(n+1)$th row and boxes labelled with $m_{ij}$.

Hence, $S\boxplus_{D}S^{*}$ is a numerical semigroup if and only if $2s_{n}-s_{i}-s_{j}\in G(S\boxplus_{D}S^{*})$ and $s_{n}-s_{i}\in G(S\boxplus_{D}S^{*})$ for all $i,j\in\{0,\dots,n-1\}$. That is, $S\boxplus_{D}S^{*}$ is a numerical semigroup if and only if $s_{n}-s_{i}\neq s_{j}$ for all $i,j \in \{0,\dots,n-1 \}$, $2s_{n}-s_{i}-s_{j}\neq s_{k}$ and $2s_{n}-s_{i}-s_{j}\neq s_{n}+1+a_{m}-a_{l}$ for all $i,j,k\in\{0,\dots,n-1\}, l=1,\dots,m$. On the other hand, for all $i,j \in \{0,\dots,n-1 \}$, we have $$s_{n}-s_{i}\neq s_{j} \Longleftrightarrow s_{n}\neq s_{i}+s_{j} \Longleftrightarrow s_{n} \text{ is a minimal generator of } S.$$
Furthermore, since $S$ is a numerical semigroup, for all $i,j\in\{0,\dots,n-1\}$ and $l=1,\dots,m$, we have $$2s_{n}-s_{i}-s_{j}\neq s_{n}+1+a_{m}-a_{l} \Longleftrightarrow s_{i}+s_{j}\neq a_{l}.$$ Therefore, $S\boxplus_{D}S^{*}$ is a numerical semigroup if and only if $s_{n}$ is a minimal generator of $S$ and $2s_{n}-s_{i}-s_{j}\neq s_{k}$ for all $i,j,k\in\{0,\dots,n-1\}$.

Proof of (2): By definition, $S\boxplus_{E}S^{*}=\{ s_{1},\dots,s_{n-1},s_{n},s_{n}+a_{m}-a_{m-1},\dots,s_{n}+a_{m}-a_{1},2s_{n},\rightarrow \}$. Following a similar way that we built in proof of (1), we get $m_{ij}=2s_{n}-s_{i-1}-s_{j-1}-1$ for all $i,j\in\{1,\dots,n\}$, and $$G(S\boxplus_{E}S^{*})=\{ a_{1},\dots,a_{k},s_{n},2s_{n}-s_{n-1}-1,\dots,2s_{n}-s_{1}-1,2s_{n}-1\}.$$ Notice that we do not have an extra row in $Y_{S\boxplus_{E}S^{*}}$. Therefore, we only need to check all $m_{ij}$.

Hence $S\boxplus_{E}S^{*}$ is a numerical semigroup if and only if $2s_{n}-s_{i}-s_{j}-1\in G(S\boxplus_{E}S^{*})$ for all $i,j\in\{0,\dots,n-1\}$. That is, $S\boxplus_{E}S^{*}$ is a numerical semigroup if and only if $2s_{n}-s_{i}-s_{j}-1\neq s_{k}$ and $2s_{n}-s_{i}-s_{j}-1\neq s_{n}+a_{m}-a_{l}$ for all $i,j,k\in\{0,\dots,n-1\}, l=1,\dots,m$. Furthermore, for all $i,j,\in\{0,\dots,n-1\}$ and $l=1,\dots,m$, we have $$2s_{n}-s_{i}-s_{j}-1\neq s_{n}+a_{m}-a_{l} \Longleftrightarrow s_{i}+s_{j}\neq a_{l}.$$ However, $s_{i}+s_{j}\neq a_{l}$ is always true since $S$ is a numerical semigroup. Therefore, $S\boxplus_{D}S^{*}$ is a numerical semigroup if and only if $2s_{n}-s_{i}-s_{j}-1\neq s_{k}$ for all $i,j,k\in\{0,\dots,n-1\}$.

Proofs for (3) and (4) are similar.

\end{proof}

\begin{corollary} \label{pr1}
	Let $S=\{ 0,s_{1},\dots,s_{n} \}$ be a numerical semigroup with $G(S)=\{ a_{1},\dots,a_{m} \}$.
	\begin{enumerate}
		\item If $2s_{n}-s_{i}-s_{j}-1\neq s_{k}$ $\forall i,j,k \in \{0,\dots,n-1\}$, then $S\boxplus_{E}S^{*}$ is a symmetric numerical semigroup.
		\item If $2a_{m}-s_{i}-s_{j}-1\neq s_{k}$ $\forall i,j,l \in \{0,\dots,n-1\}$, then $S\boxplus_{O}S^{*}$ is a symmetric numerical semigroup.
		\item If $s_{n}$ is a minimal generator of $S$, and $2s_{n}-s_{i}-s_{j}\neq s_{k}$ $\forall i,j,k \in \{0,\dots,n-1\}$, then $S\boxplus_{D}S^{*}$ is a pseudo-symmetric numerical semigroup.
		\item If $s_{n}$ is a minimal generator of $S$, and $2a_{m}-s_{i}-s_{j}\neq s_{k}$ $\forall i,j,k \in \{0,\dots,n-1\}$, then $S\boxplus_{C}S^{*}$ is a pseudo-symmetric numerical semigroup.
	\end{enumerate} 
\end{corollary} 

\begin{proof}
By Theorem \ref{thm}, we know that if we have these assumptions are satisfied, we get $S\boxplus_{E}S^{*}$, $S\boxplus_{O}S^{*}$, $S\boxplus_{D}S^{*}$ and $S\boxplus_{C}S^{*}$ are numerical semigroups.

By definitions of end-to-end and overlap sums, $F(S\boxplus_{E}S^{*})=2s_{n}-1$ and $F(S\boxplus_{O}S^{*})=2s_{n}-3$. Also by definitions of discrete and conjoint sums, $F(S\boxplus_{D}S^{*})=2s_{n}$ and $F(S\boxplus_{C}S^{*})=2s_{n}-2$. Then it follows by Lemma \ref{l3} and Proposition \ref{pps1}.
\end{proof}

\section{Ring Theoretic Correspondence}\label{sec5}

In this section, we give the ring theoretic correspondence of Corollary \ref{pr1}. Let $S$ be a numerical semigroup minimally generated by $\{ n_{1},\dots,n_{p} \}$, $\Bbbk$ be a field and $x$ be an indeterminate. Then $\Bbbk[S]=\Bbbk[x^{s} \mid s\in S]$ is a subring of the polynomial ring $\Bbbk[x]$, and it is called the semigroup ring of $S$ over $\Bbbk$. We then have $$\Bbbk[S]=\bigg\{ \sum_{0\leq s<\infty}a_{s}x^{s} \mid a_{s}=0 \text{ if } s\in G(S) \bigg\}$$ and $\mathfrak{m}=\langle x^{n_{1}},\dots,x^{n_{p}} \rangle$ is a maximal ideal of $\Bbbk[S]$.

As in the polynomail case, $\Bbbk[\![S]\!]=\Bbbk[\![x^{s} \mid s\in S]\!] $ is a subring of the power series ring $\Bbbk[\![x]\!]$, and it is usually called the semigroup ring associated to $S$. We also have $$\Bbbk[\![S]\!]=\bigg\{ \sum_{0\leq s\leq\infty}a_{s}x^{s} \mid a_{s}=0 \text{ if } s\in G(S) \bigg\}$$ is a local ring with the maximal ideal and $\mathfrak{m}=\langle x^{n_{1}},\dots,x^{n_{p}} \rangle$. Notice that $\Bbbk[\![S]\!]$ is $\mathfrak{m}$-adic completion of $\Bbbk[S]$. Next corollaries show the ring theoretic correspondences of the results in section \ref{sec4}.

\begin{corollary} 
	Let $R$ be $\Bbbk[S]$ or $\Bbbk[\![S]\!]$ for some numerical semigroup $S=\{ 0,s_{1},\dots,s_{n} \}$ with $G(S)=\{ a_{1},\dots,a_{m} \}$. If $2s_{n}-s_{i}-s_{j}-1\neq s_{k}$ or $2a_{m}-s_{i}-s_{j}-1\neq s_{k}$ for all $i,j,k \in \{0,\dots,n-1\}$, then $R$ has at least one Gorenstein subring whose value semigroup is $S\boxplus_{E}S^{*}$ or  $S\boxplus_{O}S^{*}$.
\end{corollary}

\begin{proof}
	Due to \cite{Ku}, we know that $R$ is Gorenstein if and only if $S$ is symmetric. Then by Corollary \ref{pr1}, it is straightforward.
\end{proof}

\begin{corollary}
Let $S=\{ 0,s_{1},\dots,s_{n} \}$ be a numerical semigroup with $G(S)=\{ a_{1},\dots,a_{m} \}$. If $s_{n}$ is a minimal generator of $S$, and $2s_{n}-s_{i}-s_{j}\neq s_{k}$ or $2a_{m}-s_{i}-s_{j}\neq s_{k}$ for all $i,j,k \in \{0,\dots,n-1\}$, then $\Bbbk[\![S]\!]$ has at least one Kunz subring which is $\Bbbk[\![S\boxplus_{D}S^{*}]\!]$ or $\Bbbk[\![S\boxplus_{C}S^{*}]\!]$.
\end{corollary}

\begin{proof}
Due to \cite{BDF}, we know that $\Bbbk[\![S]\!]$ is a Kunz domain if and only if $S$ is pseudo-symmetric. Then by Corollary \ref{pr1}, it is straightforward.
\end{proof}


\end{document}